\documentclass{amsart}

\usepackage[dvips]{graphicx}
\usepackage{xcolor}

\theoremstyle{plain}
\newtheorem{theorem}{Theorem}
\newtheorem{lemma}[theorem]{Lemma}
\newtheorem{corollary}[theorem]{Corollary}

\theoremstyle{definition}

\theoremstyle{remark}
\newtheorem{remark}[theorem]{Remark}
\newtheorem*{acknowledgements}{Acknowledgements}

\newcommand{\del}{\partial}

\newcommand{\N}{\mathbb{N}}

\begin{document}

\title{Strongly irreducible Heegaard splittings of hyperbolic 3-manifolds}

\author[Kalelkar]{Tejas Kalelkar}
\address{Mathematics Department, Indian Institute of Science Education and Research, Pune 411008, India}
\email{tejas@iiserpune.ac.in}

\date{\today}

\keywords{}

\subjclass[2010]{Primary 57M50, 57M99}

\begin{abstract}
Colding and Gabai have given an effective version of Li's theorem that non-Haken hyperbolic 3-manifolds have finitely many irreducible Heegaard splittings. As a corollary of their work, we show that Haken hyperbolic 3-manifolds have a finite collection of strongly irreducible Heegaard surfaces $S_i$ and incompressible surfaces $K_j$ such that any strongly irreducible Heegaard surface is a Haken sum $S_i + \sum_j n_j K_j$, up to one-sided associates of the Heegaard surfaces.
\end{abstract}

\maketitle

Colding and Gabai\cite{ColGab} used negatively curved branched surfaces that carry all index $\leq$ 1 surfaces to show that non-Haken hyperbolic 3-manifolds have finitely many irreducible Heegaard surfaces. They end with a conjecture for Heegaard splittings of Haken manifolds (Question 7.10 of \cite{ColGab}). In this short note we show that in fact their proof essentially resolves this conjecture for strongly irreducible splittings of hyperbolic manifolds, stated below as a theorem:

\begin{theorem}\label{mainthm}
Let $M$ be a closed hyperbolic manifold. There exist finitely many surfaces $S_1,..., S_n$ which are either strongly irreducible Heegaard surfaces or the one-sided associates of strongly irreducible Heegaard surfaces and finitely many incompressible surfaces $K_1, ..., K_m$ such that if $S$ is any strongly irreducible Heegaard surface then for some $S_i$ either $S=S_i+\sum_j n_j K_j$ or it has a one-sided associate $S'$ and $S'=S_i +\sum_j n_j K_j$. Furthermore, if the surface $K=\sum_j n_j K_j$ is non-empty then it is incompressible.
\end{theorem}

A restatement of the Claim in the proof of Theorem 6.3 of \cite{ColGab} is the following:
\begin{theorem}\label{CG}
Let $M$ be a closed hyperbolic manifold. There exists $K^* >0$ and finitely many pairs of branched surfaces $(E_i, L_i)$ with $L_i$ a (possibly empty) sub-branch surface of $E_i$ such that for each strongly irreducible Heegaard surface $S$ of $M$ there exists a pair $(E, L)$ from this finite collection such that
\begin{enumerate}
\item{$S$ or an associated one-sided surface $S'$ is fully carried by $E$.}
\item{$S=S_1 + S_2$ or $S'=S_1 + S_2$, where $S_2$ is fully carried by $L$ and $S_1 = \sum n_i F_i$ where $F_i$ are the fundamental surfaces carried by $E$ and $0 \leq n_i < K^*$.}
\item{If $L \neq \emptyset$ then $L$ is incompressible, carries no torus, has no complementary monogons and has no disks of contact.}
\end{enumerate}
\end{theorem}
\begin{proof}
Take the finite collection $(E_i, L_i)$ of branch surfaces as in the Claim in the proof of Theorem 6.3 of \cite{ColGab}. As $L_i$ are sub-branch surfaces of splittings of a $-1/2$-negatively curved branch surface so they carry no tori. As no component of $\del_h N(L_i)$ is a disk or annulus so if a complementary component of $N(L_i)$ has a monogon, it has a compression disk (Lemma 5.5 of \cite{ColGab}). If any $E$ carries infinitely many strongly irreducible Heegaard surfaces $H_n$ and the corresponding $L$ is compressible then the rest of the proof of Theorem 6.3 of \cite{ColGab} shows that some $H_n$ is in fact weakly reducible.
\end{proof}

In the following lemma, we use the notation $\N$ to denote non-negative integers.
\begin{lemma}\label{comb}
Given a sequence $\{H_n\}$ of $\N^k$, there exists a finite subset $\{S_i\}$ of $\{H_n\}$ such that for each $n \in \N$ there exists some $S_i$ with $H_n -S_i \geq 0$.\end{lemma}
\begin{proof}
We proceed by induction on $k$. When $k=1$, take $S$ as the minimum of $H_n$ so that $H_n - S \geq 0$. Assuming the statement for $k-1$, we ignore the $k$-th coordinate of the sequence $\{H_n\}$, so that there exist finitely many $S_i$ in the sequence such that for any $H_n$ there exists $S_i$ such that $H_n(j) \geq S_i(j)$ for $j=1...(k-1)$ where $H_n(j)$ and $S_i(j)$ denote the $j$-th component of $H_n$ and $S_i$ respectively. Let $M$ be the maximum of the finite set $\{S_i(k)\}$. For each $m<M$, consider the subsequence $\{H_n^{(m)}\}$ of $\{H_n\}$ where $H_n^{(m)}(k)=m$.  Again ignoring the last coordinate we get finitely many $S_i^{(m)}$'s in the sequence such that for all $n\in \N$ there exists some $S_i^{(m)}$ such that $H_n^{(m)}(j) \geq S^{(m)}_i (j)$ for $j=1...(k-1)$ and $H_n^{(m)}(k)=S_i^{(m)}(k)=m$. 

We claim that the union of the sets $\{S_i\}$, $\{S_i^{(1)}\}$, ..., $\{S_i^{(M-1)}\}$ is the required finite collection for $k$. For any $H_n$, ignoring the $k$-th coordinate there exist $S_i$ such that $H_n(j) \geq S_i(j)$ for $j=1...(k-1)$. If $H_n(k) \geq S_i(k)$ then $H_n - S_i \geq 0$ as required. If $H_n(k)=m < S_i(k) \leq M$ then $H_n$ is a point in the subsequence $\{H_n^{(m)}\}$ and hence there exists $S_i^{(m)}$ such that $H_n - S_i^{(m)} \geq 0$ as required.
\end{proof}
We can now prove the main theorem of this article:
\begin{proof}[Proof of Theorem \ref{mainthm}]
For each pair $(E, L$) of branched surfaces in Theorem \ref{CG}, if $S=S_1 + S_2$ is a strongly irreducible Heegaard surface carried by $E$ then there are only finitely many possibilities for $S_1$. Let $\{H^n\}$ be a sequence of all strongly irreducible Heegaard surfaces carried by $E$ with $H^n_1 = S_1$. Applying Lemma \ref{comb} to the coefficients of $H^n$ with respect the fundamental surfaces of $E$, there are finitely many $\{S^i\}$ in this sequence such that for any $H^n$ there exists some $S^i$ with $H^n - S^i = K$ a closed surface carried by $L$. By Theorem 2 of \cite{Oer}, $K$ is incompressible. Let $K_j$ be the collection of fundamental surfaces of the finitely many sub-branched surfaces $L$. Expressing $K$ in terms of these $K_j$ we get the required result.
\end{proof}

\begin{remark}
As the bound $K^*$ and the branched surfaces $(E, L)$ used in Theorem \ref{CG} are all effectively constructible so the finitely many strongly irreducible surfaces $S_i$ and incompressible surfaces $K_j$ in Theorem \ref{mainthm} are also effectively constructible.
 \end{remark}
 
In all known examples of manifolds which have infinitely many irreducible Heegaard surfaces, the Heegaard surfaces are of the form $S+nK$ \cite{Li, MorSchSed}. As a corollary of Theorem \ref{mainthm}, it is easy to observe a weak version of this phenomenon.

\begin{corollary}
Let $M$ be a closed hyperbolic manifold. If $M$ has infinitely many strongly irreducible Heegaard surfaces then it has infinitely many strongly irreducible Heegaard surfaces or their one-sided associates which are the Haken sum $S+K^i$ where $S$ is a strongly irreducible Heegaard surface or its one-sided associate and $K^i$ are incompressible surfaces all fully carried by a fixed branched surface.
\end{corollary}
\begin{proof}
Given infinitely many strongly irreducible Heegaard surfaces or their one-sided associates carried by a fixed branched surface $E$, by ignoring the fundamental surfaces of $L$ with zero coefficients we can pass to a subsequence $S^i=S+\sum_j n^i_j K_j$ where for each $j$, $n^i_j$ is a sequence of positive numbers. Let $L'$ be the smallest sub-branched surface of $E$ that carries all these $K_j$. Then each $K^i = \sum_j n^i_j K_j$ is an incompressible surface fully carried by $L'$.
\end{proof}

With respect to a triangulation, strongly irreducible splitting surfaces can be put in almost normal form and with respect to a hyperbolic metric, they can be isotoped to  either index $\leq 1$ minimal surfaces or to the double cover of an index-0 surface with an unknotted tube connecting the sheets. There is no such nice form for weakly reducible splitting surfaces, so Question 7.10 of \cite{ColGab} is still open in its full generality.

\begin{acknowledgements}
We would like to thank David Gabai for several helpful comments and suggestions. The author was partially supported by the MATRICS grant of Science and Engineering Research Board, GoI.
\end{acknowledgements}

\bibliographystyle{amsplain}

\end{document}